\documentclass{ifacconf}

\usepackage{graphicx}
\usepackage{natbib}

\usepackage{enumitem}
\usepackage{glossaries}
\newacronym{fosf}{FOSF}{finite-dimensional observer-based state feedback}
\newacronym{bc}{BC}{boundary condition}
\newacronym{rs}{RS}{Riesz-spectral}
\newacronym{bcs}{BCS}{boundary control system}
\newacronym{bos}{BOS}{system with boundary observation}
\newacronym{ode}{ODE}{ordinary differential equation}
\usepackage{amsmath}
\usepackage{amssymb}
\usepackage{amsfonts}
\usepackage{mathrsfs}

\renewcommand{\star}{\diamond}

\newcommand{\SysOp}{\Sigma}

\newcommand{\SysOpObs}{\SysOp_{\text{obs}}}

\newcommand{\EA}{\mathcal{A}}

\newcommand{\EB}{\mathcal{B}}
\newcommand{\EC}{\mathcal{C}}

\newcommand{\EK}{\mathcal{K}}
\newcommand{\EL}{\mathcal{L}}

\newcommand{\BA}{\mathfrak{A}}
\newcommand{\BB}{\mathfrak{B}}
\newcommand{\BC}{\mathfrak{C}}
\newcommand{\BD}{\mathfrak{D}}
\newcommand{\BK}{\mathfrak{K}}

\newcommand{\BR}{\mathfrak{R}}
\newcommand{\BG}{\varkappa}

\newcommand{\OA}{\mathscr{A}}

\newcommand{\OC}{\mathscr{C}}
\newcommand{\OK}{\mathscr{K}}

\newcommand{\OR}{\mathscr{R}}

\newcommand{\MA}{A}

\newcommand{\MB}{B}
\newcommand{\MC}{C}
\newcommand{\MK}{K}
\newcommand{\ML}{L}
\newcommand{\mb}{b}

\newcommand{\OKc}{\breve{\OK}}
\newcommand{\BKc}{\breve{\BK}}

\newcommand{\BKu}{\mathring{\BK}}
\newcommand{\OKu}{\mathring{\OK}}
\newcommand{\Kb}{\EK}

\newcommand{\Lu}{\mathring{\EL}}
\newcommand{\Lb}{\EL}

\newcommand{\ctrlIsText}{\text{\rm{c}}}
\newcommand{\obsIsText}{\text{\rm{o}}}
\newcommand{\desSText}{\text{\rm{d}}}
\newcommand{\desCsText}{\text{\rm{dc}}}
\newcommand{\desOsText}{\text{\rm{do}}}
\newcommand{\clSText}{\text{\rm{cl}}}

\newcommand{\ctrlIs}[2][]{#2^{\ctrlIsText #1}}
\newcommand{\obsIs}[2][]{#2^{\obsIsText #1}}
\newcommand{\desS}[2][]{#2^{\desSText #1}}
\newcommand{\desCs}[2][]{#2^{\desCsText #1}}
\newcommand{\desOs}[2][]{#2^{\desOsText #1}}
\newcommand{\clS}[2][]{#2^{\clSText #1}}

\newcommand{\Ss}{\mathcal{X}}
\newcommand{\Is}{\mathcal{U}}
\newcommand{\Os}{\mathcal{Y}}

\newcommand{\Nats}{\mathbb{N}}

\newcommand{\Reels}{\mathbb{R}}
\newcommand{\Compl}{\mathbb{C}}

\newcommand{\LOp}{\mathscr{L}}
\newcommand{\Lz}{L^{2}}
\newcommand{\lz}{l^{2}}
\newcommand{\li}{l^{\infty}}

\newcommand{\Hn}[1]{H^{#1}}

\newcommand{\eval}{\lambda}
\newcommand{\evec}{\varphi}
\newcommand{\evecu}{\varphi_{\uu}}
\newcommand{\evecy}{\varphi_{\yy}}
\newcommand{\evecI}{\evec_{\text{I}}}
\newcommand{\evecII}{\evec_{\text{II}}}

\newcommand{\dual}[3][]{\langle #2 , #3 \rangle_{D(\EA^*_{#1})}}
\newcommand{\dualf}[3]{\langle #1 , #2 \rangle_{D(#3)}}

\newcommand{\dualo}[3][]{\langle #2 , #3 \rangle_{D(\obsIs[*]{\EA}_{#1})}}
\newcommand{\dualdc}[2]{\langle #1 , #2 \rangle_{D(\desCs[*]{\EA})}}

\newcommand{\scal}[2]{\langle #1 , #2 \rangle_{\Ss}}
\newcommand{\scalf}[3]{\langle #1 , #2 \rangle_{#3}}

\newcommand{\s}{x}
\newcommand{\ms}{p}
\newcommand{\uu}{u}

\newcommand{\yy}{y}

\newcommand{\msh}{\hat{\ms}}

\newcommand{\sh}{\hat{\s}}
\newcommand{\yh}{\hat{\yy}}
\newcommand{\st}{\tilde{\s}}
\newcommand{\yt}{\tilde{\yy}}

\newcommand{\sw}{w}
\newcommand{\ca}{\alpha}
\newcommand{\cb}{\beta}
\newcommand{\cc}{\gamma}

\newcommand{\kappac}{\kappa_{\text{c}}}
\newcommand{\kappao}{\kappa_{\text{o}}}
\newcommand{\muc}{\mu_{\text{c}}}
\newcommand{\muo}{\mu_{\text{o}}}

\newcommand{\soc}{\xi}
\newcommand{\sfc}{\nu}
\newcommand{\fo}{\chi}

\newcommand{\thu}{\theta_{-}}
\newcommand{\acc}{a}

\newcommand{\ku}{k_\uu}
\newcommand{\ky}{k_\yy}
\newcommand{\cu}{c_\uu}
\newcommand{\cy}{c_\yy}
\newcommand{\cuh}{\hat{c}_\uu}
\newcommand{\cyh}{\hat{c}_\yy}
\newcommand{\MBz}{\MB_2}
\newcommand{\MBe}{\MB_1}
\newcommand{\MBn}{\MB_0}
\newcommand{\MGz}{G_2}
\newcommand{\MGe}{G_1}
\newcommand{\MGn}{G_0}

\newcommand{\diff}[1]{\partial_{#1}}

\newcommand{\CN}{$\text{C}_{0}$}

\newcommand{\Transpose}{\mathsf{T}}

\newcommand{\Resol}{R}
\newcommand{\Id}{\mathcal{I}}

\newcommand{\cconj}[1]{\overline{#1}}

\newcommand{\Dirac}[1]{\delta_{#1}}

\usepackage{accents}

\newcommand{\homm}{\text{hom}}
\newcommand{\inh}{\text{inh}}

\begin{document}
\begin{frontmatter}

\title{Late lumping of observer-based state feedback for boundary control systems\thanksref{footnoteinfo}}
\thanks[footnoteinfo]{This work has been submitted to IFAC for possible publication.}

\author[First]{Marcus Riesmeier} 
\author[First]{Frank Woittennek} 

\address[First]{Institute of Automation and Control Engineering, UMIT Tirol, Private University for Health Sciences and Technology, Eduard-Wallnöfer-Zentrum 1, 6060 Hall in Tirol, Austria (e-mail: \{marcus.riesmeier, frank.woittennek\}@umit-tirol.at).}

\begin{abstract}                
Infinite-dimensional linear systems with unbounded input and output operators are considered. For the purpose of finite-dimensional observer-based state feedback, an observer approximation scheme will be developed which can be directly combined with existing late-lumping controllers and observer output injection gains. It relies on a decomposition of the feedback gain, resp.\ observer output injection gain, into a bounded and an unbounded part. Based on a perturbation result, the spectrum-determined growth condition is established, for the closed loop.
\end{abstract}


\end{frontmatter}

\section{Introduction}
\label{sec:intro}

Infinite-dimensional linear systems with unbounded input and output operators are considered.
The purpose of the article is to combine feedback gain and observer gain approximations to \gls{fosf}.
Therefore, this article addresses the approximation of the observer.
Previous results in this direction come e.g. from \citep{Deutscher2013IJC,GrueneMeurer2022,Curtain1984siam}.

Using the perturbation result from~\cite{XuSallet1996siam},
it will be proven that the closed-loop system is
a discrete \gls{rs} system.
Hence, the stability can be checked by
computing the eigenvalues of the closed-loop operator.
To the authors' knowledge, this is the first proof in this context,
which covers both analytic and hyperbolic systems.
For analytic systems~\cite{Curtain1984siam} has proven a similar result,
using a perturbation result from~\cite{Kato1995}.

The article is organized as follows. In Section~\ref{sec:prelim} the system
class will be introduced. In Section~\ref{sec:late_lumping} the structure of the observer, and the
feedback and observer gain approximations will be stated.
In Section~\ref{sec:approx} the approximation scheme for the observer will be derived.
Section~\ref{sec:conclusion} summarizes the article.

\section{Preliminaries}
\label{sec:prelim}
Within this section the notation and the structural
properties of the systems and designs under
consideration will be introduced.

\subsection{Basic notation}
The complex conjugate of a complex number $c\in\Compl$ is denoted by
 $\cconj{c}$.
Moreover, $\Lz(a,b;\Compl^n)$ denotes
the Lebesgue space of square-integrable functions $f:[a,b]\to\Compl^n$, $z\mapsto f(z)$,
while $\Hn{n}(a,b;\Compl^n)$ is the usual Sobolev space of $n$ times weakly  differentiable (in $\Lz(a,b;\Compl^n)$) functions on $[a,b]$ taking values in $\Compl^n$. 

The partial derivative of order $n\in\Nats$ w.r.t.\ a variable $z$ is denoted by $\diff{z}^n$.
Throughout this paper, $t\in\Reels$ stands exclusively for the time variable, the first  (partial) derivative w.r.t.\ $t$ of a function $h$ is abbreviated by $\dot{h}$.
For  two Banach spaces  $\mathcal{M}$ and $\mathcal{N}$, $\LOp(\mathcal{M}, \mathcal{N})$ denotes the
Banach space of linear bounded operators $\mathcal{M}\to\mathcal{N}$.

Let $\Ss$ denote a separable Hilbert space and
$\EA:\Ss\rightarrow \Ss$ a linear operator,
which is not necessarily bounded on $\Ss$.
The spectrum and the point spectrum of $\EA$ are denoted by $\sigma(\EA)$ and $\sigma_p(\EA)$, respectively.
Furthermore, $\{\eval_i\}$ denotes the
sequence of eigenvalues of $\EA$ and $\{\evec_i\}$ the corresponding sequence of
eigenvectors. The adjoint operator of $\EA$ is denoted by~$\EA^*$, with
eigenvalues~$\{\eval_i^*\}$ and eigenvectors $\{\evec_i^*\}$.
The resolvent of $\EA$ is denoted by $\Resol(\EA,\eval)=(\eval I - \EA)^{-1}$, $\eval\not\in\sigma(\EA)$,
with $I$ the identity.
Moreover, $D(\EA)$ is the domain of $\EA$ and
$D(\EA)'$ is the dual space of $D(\EA)$. These spaces 
 are equipped with the graph norm and
the corresponding dual norm, respectively.
The duality pairing in $D(\EA^*)$ is denoted by
$\dual{F}{g},\,F\in D(\EA^*)',\,g\in D(\EA^*)$ and
the scalar product in $\Ss$ is denoted by
$\scalf{f}{g}{\Ss},\,f,g\in\Ss$. The scalar product as well as the duality
pairing take complex conjugation on the second argument.
The space of square-summable sequences and the space of bounded
sequences are denoted by $\lz$ and $\li$, respectively. 
Finally, $\Dirac{r}(\cdot)$ is the Dirac delta distribution centered in $r\in\Reels$.

\subsection{System structure}
\label{sec:prelim:sys_struct}
Boundary control systems~\citep{Fattorini1968siam} with boundary observation are
considered.
They are of the form\footnote{Note that any system given in the seemingly more general form $\dot{\s}(t) = {\BA} \s(t) + \BB\uu(t),\,\BB\in\Ss$, \eqref{eq:bc_system_input},  $\yy(t) = {\BC} \s(t) + \BD\uu(t),\,\BD\in\Compl$ can be restated as $\dot{\s}(t) = ({\BA} + \BB\BR)\s(t),\,\BB\in\Ss$, \eqref{eq:bc_system_input},  $\yy(t) = ({\BC}+\BD\BR) \s(t)$ and is, therefore, covered 
by \eqref{eq:bc_system}.}
\begin{subequations}
\label{eq:bc_system}
\begin{align}
\label{eq:bc_system_dynamics}
\dot{\s}(t) &= \BA \s(t), && \s(0) = \s_0 \in \Ss \\
\label{eq:bc_system_input}
\uu(t) &= \BR \s(t) \\
\label{eq:bc_system_output}
\yy(t) &= \BC \s(t)
\end{align}
\end{subequations}
with state $\s(t)\in\Ss$, input $\uu(t)\in\Is=\Compl$
and output $\yy(t)\in\Os=\Compl$, cf.~\citep{Fattorini1968siam}.
The state space $\Ss$ is a separable Hilbert space and
$\BA:\Ss\supset D(\BA) \rightarrow \Ss$,
$\BR:\Ss\supset D(\BA) \rightarrow \Is$ and
$\BC: \Ss\supset D(\BA)\rightarrow\Os$ are unbounded operators on $\Ss$.

It is convenient to consider~\eqref{eq:bc_system} also in the form
\begin{align}
  \label{eq:sigma_system}
(\dot{\s}(t), \yy(t)) &= \SysOp \left(\s(t),\uu(t)\right), && \s(0) = \s_0 \in \Ss,
\end{align}
where
$\SysOp:\Ss\times\Is \supset D(\SysOp) \rightarrow \Ss\times\Os$
is unbounded
and $D(\SysOp)=\{ (h_\s,h_\uu)\in D(\BA)\times\Is \,|\, \BR h_{\s}=h_{\uu} \}$
is dense in $\Ss\times\Is$.

  System~\eqref{eq:bc_system} (resp.\ \eqref{eq:sigma_system}) is called a \gls{bos},
  if the adjoint system
  \begin{align}
    \label{eq:dual_sys}
  \left(\dot\s^*(t),\uu^*(t)\right) = \SysOp^*\left(\s^*(t), \yy^*(t)\right)
  \end{align}
  with
  input $\yy^*(t)\in\Os$ and output $\uu^*(t)\in\Is$ is a \gls{bcs}, i.e.,
there exist operators
\begin{align*}
\OA:D(\OA)\to \Ss, && \OR:D(\OA)\to \Os, && \OC:D(\OA)\to \Is
\end{align*}
with similar properties
as  $\BA$, $\BR$, $\BC$ such that
\begin{multline*}
\SysOp^*(\s^*(t),\yy^*(t))=(\OA \s^*(t),\OC \s^*(t)), \\
    D(\SysOp^*)=\{(h_{\s^*},h_{\yy^*})\in D(\OA)\times\Os \,|\, \OR h_{\s^*}=h_{\yy^*} \}.
\end{multline*}

For a unified treatment of the controller and observer design, a reformulation of
\eqref{eq:sigma_system} (resp.\ \eqref{eq:bc_system}) as evolution equation
\begin{subequations}
\label{eq:system}
\begin{align}
\label{eq:system_dynamics}
\dot{\s}(t) &= \EA \s(t) + \EB \uu(t), && \s(0) = \s_0 \in \Ss,
\end{align}
\end{subequations}
as described in~\cite[Chapter 3]{BensoussanPratoDelfourMitter2007}, is considered.
Therein, the system operator $\EA:\Ss\supset D(\EA) \rightarrow \Ss$
and the input operator $\EB\in\LOp(\Is, D(\EA^*)')$
are defined by
the following relations: 
\begin{subequations}
  \label{eq:bc_ev_relation}
\begin{align}
&D(\EA) = \{h \in D(\BA) \,| \, \BR h = 0 \} &&\\
&\EA h = \BA h,&&\forall\, h\!\in\! D(\EA) \\
  \label{eq:bc_ev_relation_input}
&\dualf{\EB}{h}{\EA^{*}}=\scalf{\left(0, 1\right)}{\SysOp^*(h,0)}{\Ss\times\Is}, && \forall\,h\!\in\! D(\EA^*).
\end{align}
\end{subequations}

Similarly, the adjoint system \eqref{eq:dual_sys} is associated with
\begin{subequations}
\label{eq:system_dual}
\begin{align}
\label{eq:system_dual_dynamics}
\dot{\s}^*(t) &= \EA^* \s^*(t) + \EC^* \yy^*(t), && \s^*(0) = \s^*_0 \in \Ss
\end{align}
\end{subequations}
where
\begin{align*}
&D(\EA^*) = \{h \in D(\OA) \, | \, \OR h = 0 \} &&\\
&\EA^* h= \OA h,&&\forall\,h\in D(\EA^*)\\
&\dualf{\EC^*}{h}{\EA}=\scalf{\left(0, 1\right)}{\SysOp (h,0)}{\Ss\times\Os}, &&\forall\,h\in D(\EA).
\end{align*}

While $\yy(t)=\EC\s(t)=\dualf{\EC^*}{\s(t)}{\EA}$ holds for the autonomous system ($\uu(t)=0$),
for the actuated system $\yy(t)\neq\EC\s(t)$ in general case, compare~\citep{Weiss94p1}. Therefore, throughout this paper the output
equation~\eqref{eq:bc_system_output} is used to determine the output of the system.

Throughout this contribution, $\EA$ is assumed to be the infinitesimal generator 
of a \CN-semigroup on $\Ss$, while both  $\EB$ and $\EC^*$ are not required to be admissible\footnote{
Instead of admissibility of the input and output operators admissibility of
the feedback and observer gain operators is required within this contribution, cf.~\citep{Rebarber1989ieee}.}
in the sense of~\citep{TucsnakWeiss2009}.

\subsubsection{Example system.}
\label{sec:example_system}
The following example system will be used in Section~\ref{sec:approx}
for the application of the proposed approximation scheme.
Therefore, for this example, in Section~\ref{sec:late_lumping}
also the corresponding feedback, observer gain and observer approximations are given.

Consider the hyperbolic system
\begin{subequations}
\label{eq:ex_sys}
\begin{align}
\begin{split}
\label{eq:ex_dgl1}
\diff{t}\sw_1(z,t) &= \ca \diff{z}\sw_2(z,t)
\end{split} \\
\begin{split}
\label{eq:ex_dgl2}
\diff{t}\sw_2(z,t) &= \cb \diff{z}\sw_1(z,t)
\end{split} \\
\begin{split}
\label{eq:ex_dgl3}
\diff{t}\sw_3(t) &= \cc \sw_2(0,t)
\end{split}
\end{align}
with \glspl{bc}
\begin{align}
\label{eq:ex_bc}
\sw_3(t) = \sw_1(0,t), && \uu(t) = \sw_2(1,t),
\end{align}
\end{subequations}
input $\uu(t)$, and output $\yy(t) = \sw_1(1,t)$. This model can be used
to describe the linearized dynamics of an undamped pneumatic system \citep{Gehring2018mathmod}.
With state
\[\s(t)=(\sw_1(\cdot,t),\sw_2(\cdot,t),\sw_3(t))\in\Ss=\Lz(0,1;\Compl^2)\times\Compl,\]
\eqref{eq:ex_sys} can be written in the form \eqref{eq:bc_system} with
$\BR h \! =  \! h_2(1)$,
$\BC h \! =  \! h_1(1)$,
$\BA h \! =  \! (\ca\diff{z}h_2,\, \cb\diff{z}h_1,\, \cc h_2(0))$, $(h_1,h_2,h_3)\in D(\BA)$,
\begin{align*}
D(\BA)  \! &=  \! \{(h_1, h_2, h_3) \!\in \!\Ss|\diff{z}h_1, \diff{z}h_2\!\in\!\Lz(0,1),h_3\!=\!h_1(0)\}
\end{align*}
or in state space representation~\eqref{eq:system} with $\EB = (\ca\,\Dirac{1}(\cdot),0,0)$,
\begin{align*}
\EA\,h &= \BA\,h, \,\, h\in D(\EA) = \{(h_1, h_2, h_3) \in D(\BA)\,|\,h_2(1)=0\}.
\end{align*}
\section{Design and approximation}
\label{sec:late_lumping}
The aim of the paper is to connect the approximation schemes, proposed in
\citep{RiesmeierWoittennek2022arxiv1}, for state feedback and observer output-injection,
to an approximation scheme for \gls{fosf}.
Therefore, the basic approximation schemes from there will be briefly summarized
in the following. Based on this, a representation of the observer will be derived, which
serves as the basis for the approximation in Section~\ref{sec:approx}.

\subsection{Feedback design and approximation}
\label{sec:late_lumping_coc}
As described in~\citep{RiesmeierWoittennek2022arxiv1} the control law
\begin{align*}
  \uu(t) = \BKc\s(t) + \BG\,\desCs{\uu}(t),
\end{align*}
with arbitrary non-zero $\BG\in\Reels$ and feedback gain
\begin{align*}
  \BKc = \BR - \BG\desCs{\BR},
\end{align*}
assigns the desired dynamics $\desCs{\EA}$
to the closed loop
\begin{subequations}
\begin{align}
\dot{\s}(t) &= \BA \s(t), && \s(0) = \s_0 \in \Ss \\
\desCs{\uu}(t) &= \desCs{\BR} \s(t) \\
\yy(t) &= \BC \s(t).
\end{align}
\end{subequations}
For the \gls{bcs} $(\BA,\desCs{\BR},\BC)$
the system operators $(\desCs{\SysOp},\desCs{\EA},$ $\desCs{\EB})$ can be derived in the same way
as $({\SysOp},{\EA}, {\EB})$
are derived from $(\BA,{\BR},\BC)$, c.f.\ Section~\ref{sec:prelim:sys_struct}.

Since only the bounded part of the feedback is subject to approximation,
the feedback gain requires a decomposition
\begin{equation}
  \label{eq:fb_decomp}
  \uu(t) = \BKu\s(t) + \Kb\s(t) + \BG\,\desCs{\uu}(t)
\end{equation}
into an unbounded part $\BKu\in\LOp(D(\BA),\Is)$ and a bounded
part $\Kb\in\LOp(\Ss,\Is)$.
Therewith, the controller intermediate system
\begin{subequations}
\begin{align}
\dot{\s}(t) &= \BA \s(t), && \s(0) = \s_0 \in \Ss \\
\ctrlIs{\uu}(t) &= \ctrlIs{\BR} \s(t), && \ctrlIs{\BR} = \BR - \BKu \\
\yy(t) &= \BC \s(t)
\end{align}
\end{subequations}
can be introduced by defining
the respective input by
\begin{align*}
\ctrlIs{\uu}(t) = \Kb\s(t) + \BG\,\desCs{\uu}(t).
\end{align*}
As described above, the controller intermediate system operators
$(\ctrlIs{\SysOp}, \ctrlIs{\EA}, \ctrlIs{\EB})$ can be derived from
$(\BA,\ctrlIs{\BR},\BC)$.

For the results of this article, approximations with respect to the
eigenvectors $\{\obsIs{\evec}_1\}_1^n$ and $\{\obsIs[*]{\evec}_1\}_1^n$
of $\obsIs{\EA}$ resp. $\obsIs[*]{\EA}$ (to be introduced in the next subsection), are considered.
As described in~\citep{RiesmeierWoittennek2022arxiv1} $\Kb$ can now be approximated:
\begin{align}
  \label{eq:fb_approx}
\Kb^{n}&= \sum_{i=1}^{n}\scal{\cdot}{\obsIs[*]{\evec}_i}\, \obsIs{k}_i \,\in \,\LOp(\Ss,\Is), &&  \obsIs{k}_i = \Kb\obsIs{\evec}_i.
\end{align}

\subsubsection{State-feedback approximation for the example system.}
\label{sec:example_ctrl}
According to~\citep{RiesmeierWoittennek2022arxiv1}, the feedback gain which assigns the desired dynamics of the delay differential equation
\begin{align*}
\dot\fo(t+\tau)+\muc\,\dot\fo(t-\tau) + \kappac(\fo(t+\tau)+\muc\,\fo(t-\tau)) = 0,
\end{align*}
to the closed loop (approximately), is given by
  $\uu(t) = \BKu\s(t) + \Kb^n\s(t)$,
where the unbounded part is reads
\begin{align*}
  \BKu h = \frac{\cb\tau(\muc - 1)}{\muc + 1} h_1(1), && h=(h_1,h_2, h_3)\in D(\BA)
\end{align*}
and the bounded part $\Kb^n$, defined by~\eqref{eq:fb_approx}, is completed by the gains
$\{\obsIs{k}_i = \obsIs{\evec}_{i,3} \Kb_\sfc \obsIs{\evec}_\sfc \, | \, \obsIs{\evec}_\sfc=\theta\mapsto e^{\obsIs{\eval}_i \theta}\}_1^n$,
with
\begin{align*}
  \Kb_\sfc h &= \frac{\cb\cc\tau - \kappac}{\cc(\muc + 1)}h(\tau)
+\frac{\muc(\cb\cc\tau+\kappac)}{\cc(\muc + 1)}h(-\tau), && h\in\Ss_\sfc.
\end{align*}

\subsection{Observer design and approximation}
\label{sec:late_lumping_ooc}
In contrast to the observer gain approximation scheme given in~\citep{RiesmeierWoittennek2022arxiv1},
now an observer for the actuated system will be derived. Therefore, similarly to the controller intermediate
system $\ctrlIs{\SysOp}$ and the desired controller system $\desCs{\SysOp}$, in the following the
observer intermediate system $\obsIs{\SysOp}$ and the desired observer system
$\desOs{\SysOp}$ will be introduced.

Starting from the adjoint system~\eqref{eq:dual_sys} with system operator $\Sigma^*$,
one can state, analogously to Section~\ref{sec:late_lumping_coc}, the control law
\begin{align*}
  \yy^*(t) = \OKc\s^*(t) + \BG\,\desOs[*]{\yy}(t),
\end{align*}
with feedback gain
\begin{align*}
  \OKc = \OR - \BG\desOs{\OR}.
\end{align*}
This assigns the desired dynamics $\desOs{\EA}$ to the adjoint system.
As for the controller design, using the decomposition
\begin{align*}
  \yy^*(t) = \OKu\s^*(t) + \Lb^*\s^* + \BG\,\desOs[*]{\yy}(t),
\end{align*}
of the feedback gain $\OKc$ into an unbounded part $\OKu\in\LOp(D(\OA),\Os)$ and a bounded
part $\Lb^*\in\LOp(\Ss,\Os)$,
one can define the observer intermediate system operators
$(\obsIs[*]{\SysOp}, \obsIs[*]{\EA}, \obsIs[*]{\EC})$ and the desired operators
$(\desOs[*]{\SysOp}, \desOs[*]{\EA}, \desOs[*]{\EC})$
in terms of $(\OA,\OR,\OC,\OKu,\Lb^*)$. This is conducted in the same way as $(\ctrlIs{\SysOp}, \ctrlIs{\EA}, \ctrlIs{\EB})$
and $(\desCs{\SysOp},\desCs{\EA}, \desCs{\EB})$ are defined, in Section~\ref{sec:late_lumping_coc}, in terms of
$(\BA,\BR,\BC,\BKu,\Kb)$.

For the observer design, the system
\begin{subequations}
  \label{eq:des_cl_sys}
\begin{align}
  \label{eq:des_cl_sys_sys}
(\dot{\s}(t), \yy(t)) &= {\SysOp} \left(\s(t),\uu(t)\right),
\end{align}
is considered, together with the desired observer error system
\begin{align}
  \label{eq:des_err_sys}
  \begin{split}
(\dot{\st}(t), \yt(t)) \! &=  \!\desOs{\SysOp} \left(\st(t),0\right)
 \!= \!  \obsIs{\SysOp} \left(\st(t),0\right)  \!+ \! (\Lb\yt(t), 0),
  \end{split}
\end{align}
\end{subequations}
where $\st = \sh - \s$ is the observer error, $\yt = \yh -  \yy$ is the output error and
$\yh(t)$ is the observer output.
Equation~\eqref{eq:des_err_sys} can be rearranged to
\begin{align}
  \label{eq:observer_sigma}
  \begin{split}
(\dot{\sh}(t), \yh(t))  \!&= \! \obsIs{\SysOp} \left(\st(t),0\right)  \!+  \!{\SysOp} \left(\s(t),\uu(t)\right)  \!+  \!(\Lb\yt(t), 0) \\
 &=  \!\obsIs{\SysOpObs} \left(\sh(t),\uu(t),\yy(t)\right)  \!+  \!(\Lb\yt(t), 0),
  \end{split}
\end{align}
where $\obsIs{\SysOpObs}$ is informally defined by
\begin{align*}
\obsIs{\SysOpObs} \left(\sh(t),\uu(t),\yy(t)\right) =
\obsIs{\SysOp} \left(\st(t),0\right) + {\SysOp} \left(\s(t),\uu(t)\right).
\end{align*}
In fact, $\obsIs{\SysOpObs}$ is again a system with unbounded control action, through $\uu(t)$ and $\yy(t)$.
For a formal definition of $\obsIs{\SysOpObs}$, the adjoint of \eqref{eq:des_cl_sys}, can be employed.
Due to a matter of space, the details of this formal definition can not be included here.
Furthermore, for the purpose of observer approximation, \eqref{eq:observer_sigma}
has the state space representation:
\begin{align}
  \label{eq:ss_obs}
\dot{\sh}(t) = \obsIs{\EA} \sh(t)+\obsIs{\EB}\uu(t)- \Lu \yy(t) + \Lb\yt(t),
\end{align}
where $\obsIs{\EB}$ is the input operator derived from $\obsIs{\SysOp}$, in the same way as $\EB$ from $\SysOp$.
$\Lu$ is the adjoint of the restriction of $\OKu$ to $D(\obsIs[*]{\EA})$: $\Lu^* h = \OKu h$, $h\in D(\obsIs{\EA})$.

As described in~\citep{RiesmeierWoittennek2022arxiv1},
the bounded part $\Lb$, of the observer gain, can now be approximated:
\begin{align}
  \label{eq:obs_fb_approx}
\Lb^{n}&=\sum_{i=1}^n \obsIs{l}_i\,\obsIs{\evec}_i\,\in\,\Ss, && \obsIs{l}_i = \scal{\Lb}{\obsIs[*]{\evec}_i}.
\end{align}

\subsubsection{Observer for the example system.}
\label{sec:example_obs}
According to~\citep{RiesmeierWoittennek2022arxiv1}, the observer~\eqref{eq:observer_sigma} (resp.~\eqref{eq:ss_obs}) that approximately assigns the desired dynamics, of the delay differential equation
\begin{align*}
\dot{\yt}(t+\tau)+\muo\,\dot{\yt}(t-\tau) + \kappao(\yt(t+\tau)+\muo\,\yt(t-\tau)) = 0,
\end{align*}
to the observer error system, is defined by 
\begin{align*}
  \Lu=\rho\obsIs{\EB}, \!&& \obsIs{\EB}=(\ca\,\Dirac{1}(\cdot),0,0)\in D(\obsIs{\EA})', \!&& \rho=\frac{\cb\tau(\muo - 1)}{\muo + 1}
\end{align*}
and the approximation~\eqref{eq:obs_fb_approx} of the bounded part, which is completed by
  $\obsIs{l}_i\!=\!-\cconj{\obsIs[*]{\evec_{\soc,i,2}}(\tau)}\!-\!\dualo[\soc]{\desOs{\acc}}{\obsIs[*]{\evec_{\soc,i}}}$,
with $\desOs{\acc} \!= \!\big(\kappao ( 1 + \muo), \kappao + \muo\,\Dirac{\tau}(\cdot)\big) \in D(\obsIs[*]{\EA}_\soc)'$ and
\begin{align*}
\obsIs[*]{\evec_{\soc,i}}\!&=\!\left(\!1\!+\!\! \int_{\thu}^0 e^{\cconj{\obsIs{\eval}_i}\,\theta}\,\cconj{\obsIs{\eval}_i}\,d\theta\!\right)^{\!\!\!-1}\!\!
  \frac{\cb\tau - \rho}{2\cb\tau}\obsIs[*]{\evec_{\soc,i,3}}\,\,(1,\theta\mapsto e^{\cconj{\obsIs{\eval}_i}\,\theta}\,\cconj{\obsIs{\eval}_i}).
\end{align*}
Above, the observer intermediate system operator $\obsIs{\EA}$ results from the restriction of $\BA$ to \[D(\obsIs{\EA})=\{h\in D(\BA)|\obsIs{\BR} h = 0\},\quad \obsIs{\BR}=\BR-\rho \BC.\]

\subsection{Properties of the involved operators}
The results of this article are restricted to desired operators
$\desS{\EA}\in\{\desCs{\EA}, \desOs{\EA}\}$ which satisfy the following
assumption.
\begin{assum}
  \label{hypo:ad_spec}
  $\desS{\EA}$ has the following spectral properties.
\begin{enumerate}[ref={A\ref{hypo:ad_spec}.\arabic*},label={A\ref{hypo:ad_spec}.\arabic*:},leftmargin={3.9em}]
\item\label{item:spec_riesz} $\desS{\EA}$ is a \gls{rs}
operator~\citep{GuoZwart2001report}.
\item\label{item:spec_disc} $\desS{\EA}$ is a discrete operator \citep{DunfordSchwartz3}.
\item\label{item:spec_simple} The eigenvalues $\{\desS{\eval}_i\}_{i=1}^\infty$  of $\desS{\EA}$ are simple\footnote{Note that Assumption~\ref{item:spec_simple} is a reasonable technical assumption
in order to avoid the introduction of generalized eigenvectors and, this way, simplify computations.}.
\end{enumerate}
\end{assum}

The eigenvectors $\{\desCs{\evec}_i\}$ and $\{\desCs[*]{\evec}_i\}$ of $\desCs{\EA}$ and $\desCs[*]{\EA}$
are assumed to be normalized, such that $\scal{\desCs{\evec}_i}{\desCs[*]{\evec}_i}=1$, $i\ge 1$.
Furthermore, $\scal{\desCs{\evec}_i}{\desCs[*]{\evec}_j}=0, \, i\neq j$ follows from Assumption~\ref{hypo:ad_spec}.
In order to use a perturbation result from \citep{XuSallet1996siam} the
input operator must satisfy the following condition.
\begin{assum}
  \label{hypo:xu_sallet_h3}
Let $d_i,\,i\in\Nats$ be the distance from the eigenvalue $\desCs{\eval}_i\in\sigma(\desCs{\EA})$ to the
rest of the spectrum $\sigma(\desCs{\EA})$,
$\desCs{D}_i=\{z\in\Compl\,|\,\frac{d_i}{3}>|z-\desCs{\eval}_i|\}$ the disk centered at
$\desCs{\eval}_i$ and
$\desCs{D} = \bigcup_{i=1}^\infty \desCs{D}_i$ the union of these disks.
  The coefficients $\desCs{\mb}_i=\dualdc{\desCs{\EB}}{\desCs[*]{\evec}_i}$ ($i\in\Nats$) 
    of the modal expansion of the input operator and the eigenvalues $\desCs{\eval}_i$ ($i\in\Nats$) satisfy
    \begin{align*}
    \sum_{i=1}^\infty \left|\frac{\desCs{\mb}_i}{\eval - \desCs{\eval}_i}\right|^2 \le M < \infty,
    && \forall \eval\not\in D=\desCs{D}.
    \end{align*}
    for an appropriately chosen positive constant $M$.
\end{assum}
\section{Observer approximation}
\label{sec:approx}

In this section, a modal approximation scheme for the observer will be developed, which ensures that the boundary action of $\uu(t)$ and $\yy(t)$
is directly taken into account in the approximation.
Furthermore, the \gls{rs} property of the closed loop,
with \gls{fosf}, will be shown.

\subsection{Modal observer approximation}
\label{sec:approx_mod}

Theorem~1 from~\citep{XuSallet1996siam} ensures, due to Assumption~\ref{hypo:ad_spec}, that $\obsIs{\EA}$ possesses the spectral properties
\ref{item:spec_riesz} and \ref{item:spec_disc}, but not necessarily \ref{item:spec_simple}.
Therefore, $\obsIs{\EA}$ may have non-simple eigenvalues.
In order to avoid generalized eigenvectors
the following assumption is made.
\begin{assum}
  \label{assum:simple_ev}
  The operator $\obsIs{\EA}$ of the observer intermediate system has only simple eigenvalues.
\end{assum}
In the sequel, the eigenvectors
$\{\obsIs{\evec}_i\}$ and $\{\obsIs[*]{\evec}_i\}$ of $\obsIs{\EA}$ and $\obsIs[*]{\EA}$
are assumed to be normalized, such that $\scal{\obsIs{\evec}_i}{\obsIs[*]{\evec}_i}=1$, $i\ge 1$.

To preserve the correct output equation
of the observer~\eqref{eq:observer_sigma} the approximation has to be written on the dense subspace $D(\SysOpObs)$ of the product space $\Ss\times\Is\times\Os$. This way, the inhomogeneous boundary conditions involving $\uu(t)$ and $\yy(t)$ are directly taken care of in the approximation procedure.
This is achieved using the following ansatz for $\sh$:
\begin{align}
  \label{eq:inh_approx}
\sh_\inh(t) = \sum_{i=1}^{n}\ms_i^n(t)\obsIs{\evec}_i + \uu(t) \evecu +\yy(t) \evecy.
\end{align}
Therein, $\evecu$ and $\evecy$ are defined by 
  $\evecu = \Resol(\obsIs{\EA},\eval_\uu)\obsIs{\EB}$ and $\evecy = -\Resol(\obsIs{\EA},\eval_\yy)\Lu$,
for some (fixed) $\eval_\uu,\eval_\yy\not\in\sigma(\obsIs{\EA})$. They satisfy
\begin{align*}
  (\evecu, 1, 0)\in D(\SysOpObs), && (\evecy, 0, 1)\in D(\SysOpObs).
\end{align*}
Since the injections of $\uu(t)$ and $\yy(t)$ are already considered in $\sh_\inh(t)$,
the approximation scheme can be derived from the following duality pairing
\begin{multline*}
\dualo{\dot{\sh}_\inh(t)}{\obsIs[*]{\evec}_j} = 
\dualo{\obsIs{\EA} \sh_\inh(t)+ \Lb\yt(t)}{\obsIs[*]{\evec}_j}.
\end{multline*}
This results in the differential equation
\begin{subequations}
\label{eq:obs_mod_approx_orig}
\begin{multline}
\dot{\ms}(t) + \MBe\dot{\uu}(t) + \MGe\dot{\yy}(t) = \\
\obsIs{\MA} \ms(t) + \MBn\uu(t) + \MGn\yy(t) + \ML\yt(t)
\end{multline}
and the output equation
\begin{align}
\yh(t) = \MC^\Transpose \ms(t) + \cu \uu(t) + \cy \yy(t),
\end{align}
\end{subequations}
where $\ms(t) = (\ms_1^n(t),...,\ms_n^n(t))^\Transpose\in\Compl^n$ and
\begin{align*}
\obsIs{\MA}&=\text{diag}(\obsIs{\eval}_1,...,\obsIs{\eval}_n), && \obsIs{\MA}\in \Compl^{n\times n} \\
(\MBe)_i&=\scal{\evecu}{\obsIs[*]{\evec}_i}, && \MBe\in\Compl^n \\
(\MBn)_i&=\dualo{\obsIs{\EA}\evecu}{\obsIs[*]{\evec}_i}, && \MBn\in\Compl^n \\
(\MGe)_i&=\scal{\evecy}{\obsIs[*]{\evec}_i}, && \MGe\in\Compl^n \\
(\MGn)_i&=\dualo{\obsIs{\EA}\evecy}{\obsIs[*]{\evec}_i}, && \MGn\in\Compl^n \\
(\ML)_i&=\scal{\Lb}{\obsIs[*]{\evec}_i}, && \ML\in\Compl^n\\
(\MC)_i&=(\SysOpObs(\obsIs{\evec}_i,0,0))_2, && \MC\in\Compl^n\\
  \cu &= (\SysOpObs(\evecu,1,0))_2 \\
  \cy &= (\SysOpObs(\evecy,0,1))_2.
\end{align*}
Applying the generalized state transform
\begin{align*}
  \ms(t) &= \msh(t) - \MBe\uu(t) - \MGe\yy(t)
\end{align*}
the observer approximation~\eqref{eq:obs_mod_approx_orig} appears in the form
\begin{subequations}
\label{eq:obs_mod_approx}
\begin{align}
\label{eq:obs_mod_approx_ode}
  \dot{\msh}(t) &= \obsIs{\MA} \msh(t) + \MBz\uu(t) + \MGz\yy(t) + \ML\yt(t) \\
\label{eq:obs_mod_approx_out}
  \yh(t) &= \MC^\Transpose \msh(t) + \cuh \uu(t) + \cyh \yy(t),
\end{align}
\end{subequations}
with
\begin{align*}
  \MBz =& \MBn - \obsIs{\MA}\MBe, & \MGz &= \MGn - \obsIs{\MA}\MGe, \\
  \cuh =& \cu - \MC^\Transpose\MBe, & \cyh &= \cy - \MC^\Transpose\MGe.
\end{align*}

\subsection{Finite-dimensional observer-based state feedback}
\label{sec:approx_fb}

By using $\sh_\inh(t)$ for the \gls{fosf},
the feedback reads
\begin{align*}
 \uu(t)\! &= \!\BKu\s(t) \!+\! \Kb\sh_\inh(t) \!
 = \!\BKu\s(t) \!+ \!\MK^\Transpose\ms(t) \!+ \!\ku\uu(t) \!+ \!\ky\yy(t) \\
 &= \!\BKu\s(t) \!+ \!\MK^\Transpose\msh(t) \!+ \!(\ku \!- \!\MK^\Transpose\MBe)\uu(t)
 \!+ \!(\ky \!- \!\MK^\Transpose\MGe)\yy(t),
\end{align*}
with
\begin{align*}
  (\MK)_i = \Kb\obsIs{\evec}_i=\obsIs{k}_i,\,\MK\in\Compl^n, &&
  \ku = \Kb\evecu, &&
  \ky = \Kb\evecy.
\end{align*}
For the purpose of implementation,
this has to be rearranged to
\begin{align}
  \label{eq:inh_fb}
 \uu(t) =  \BKu_\inh\s(t) + \MK_\inh\msh(t),
\end{align}
with
\begin{align*}
  \BKu_\inh &= (1-\ku + \MK^\Transpose\MBe)^{-1}\big(\BKu + (\ky - \MK^\Transpose\MGe)\BC\big) \\
  \MK_\inh &= (1-\ku + \MK^\Transpose\MBe)^{-1}\MK^\Transpose.
\end{align*}
Obviously, the unbounded part $\BKu_\inh$ of this feedback
differs from $\BKu$. However, in many cases like in the hyperbolic
case, described in~\cite[Definition~6.3]{RiesmeierWoittennek2022arxiv1},
the unbounded part of the feedback determines the eigenvalue asymptotics.
Especially in the case of an admissible input operator,
it is not possible to change the eigenvalue asymptotics by bounded linear feedback~\citep[cf.][]{XuSallet1996siam}.
Therefore, $\sh_\inh(t)$ should not be used
within approximated feedback, at least in the general case.

To circumvent the above-described problem, the standard modal approximation
\begin{align}
  \label{eq:hom_approx}
\sh_\homm(t) = \sum_{i=1}^{n}{\msh_i^n}(t)\obsIs{\evec}_i \in D(\obsIs{\EA})
\end{align}
can be used for \gls{fosf}.
It can be easily verified that the weights, required for $\sh_\homm(t)$ are already available,
since they correspond to the state elements of the finite-dimensional observer~\eqref{eq:obs_mod_approx}:
$\msh(t) = (\msh_1^n(t),...,\msh_n^n(t))^\Transpose\in\Compl^n$.
That $\sh_\homm(t)$ cannot exactly represent the action of $\uu(t)$ and $\yy(t)$ at the boundary
in~\eqref{eq:ss_obs} is not a problem when approximating the bounded operator $\Kb\sh$ via $\Kb\sh_\homm$.
With the resulting \gls{fosf}
\begin{align}
  \label{eq:hom_fb}
  \uu(t) = \BKu\s(t) + \Kb\sh_\homm(t) = \BKu\s(t) + \MK^\Transpose\msh(t)
\end{align}
the unbounded part $\BKu$ of the underlying decomposition $\BKc = \BKu + \Kb$ will be preserved.

\subsection{Properties of the closed-loop system}
\label{sec:approx_cl}
At this point, all ingredients for the \gls{fosf} are prepared.
The closed-loop system consist of the 
system~\eqref{eq:sigma_system}, the observer~\eqref{eq:obs_mod_approx}
and the feedback~\eqref{eq:hom_fb}.
Therewith, the system operator $\clS{\EA}$ of the closed loop
\begin{align}
  \label{eq:closed_loop_system}
\begin{pmatrix}\dot{\s}(t) \\ \dot{\msh}(t)\end{pmatrix}
=
\clS{{\EA}}
\begin{pmatrix}\s(t) \\ \msh(t)\end{pmatrix}, && \begin{pmatrix}\s(t) \\ \msh(t)\end{pmatrix} \in \tilde{\Ss}=\Ss\times\Compl^n
\end{align}
is given by
\begin{multline*}
  \clS{\EA}=
  \begin{pmatrix}
  \BA & 0 \\
  \clS{\EA}_3 & \clS{\EA}_4
  \end{pmatrix}:\tilde{\Ss}\supset D(\clS{\EA}) \to \tilde{\Ss},\\
D(\clS{\EA})\! = \!
\{(h_\s, h_{\msh})\!\in \!D(\BA)\!\times\!\Compl^n | \BR h_\s\! =\! \BKu h_\s\! +\! \MK^\Transpose h_{\msh} \}
\end{multline*}
with
\begin{align*}
  \clS{\EA}_3 &= \MBz\BK + \ML\hat\BC + \MGz\BC, & 
  \hat\BC &= \cuh{\BK} + \cyh\BC \\
  \clS{\EA}_4 &= \obsIs{\MA} + \MBz\MK^\Transpose + \ML\hat\MC^\Transpose, &
  \hat\MC &= \MC + \cuh{\MK}.
\end{align*}

However, for subsequent considerations, it is useful to
consider $\clS{\EA}$ in the form
\begin{align*}
  \clS{\EA}=
  \underbrace{
    \begin{pmatrix}
      \desCs{\EA} & 0                            \\
      \clS{\hat\EA}_3 + \tilde{\MB}\Kb  & \clS{\EA}_4 - \tilde{\MB}\MK^\Transpose
    \end{pmatrix}
  }_{\tilde{\EA}}
  +
  \underbrace{
    \begin{pmatrix}
      \ctrlIs{\EB} \\ \tilde{\MB}
    \end{pmatrix}
  }_{\tilde{\EB}}
  \underbrace{
    \begin{pmatrix}
      -\Kb & \MK^\Transpose
    \end{pmatrix}
  }_{\tilde{\Kb}},
\end{align*}
with
\begin{align*}
  \clS{\hat\EA}_3 h = \clS{\EA}_3 h, && h\in  D(\clS{\hat{\EA}}_3) =D(\desCs{\EA}),
\end{align*}
$\tilde{\EA}:\tilde{\Ss}\supset D(\tilde{\EA})\rightarrow \tilde{\Ss}$,
$D(\tilde{\EA})=D(\desCs{\EA})\times\Compl^n$, input operator $\tilde{\EB}\in D(\tilde{\EA})'$
and bounded feedback $\tilde{\Kb}\in\LOp(\tilde\Ss,\Is)$.
Moreover, $\tilde{\MB}$ is an arbitrary vector
from $\Compl^n$ that places the eigenvalues of $\tilde{\EA}_4 = \clS{\EA}_4 - \tilde{\MB}\MK^\Transpose \in \Compl^{n\times n}$
such that, they are simple and distinct from $\sigma(\desCs{\EA})$.
\begin{thm}
\label{thm:foa_generation}
The closed-loop operator $\clS{\EA}$ is a generator of a \CN-semigroup,
is \gls{rs} and has compact resolvent.
\end{thm}
\begin{proof}
$\desCs{\EA}=\ctrlIs{\EA} + \ctrlIs{\EB}\Kb$ and $\tilde{\EA}_4$ have the spectral properties~\ref{item:spec_riesz}-\ref{item:spec_simple}.
Since $\tilde{\EA}$ is closed and the eigenvectors form a Riesz-basis of $\tilde{\Ss}$,
also $\tilde{\EA}$ has the spectral properties~\ref{item:spec_riesz}-\ref{item:spec_simple}.
Therewith, it follows from \cite[Lemma~3.1]{RiesmeierWoittennek2022arxiv1},
that the control system $(\tilde{\EA}, \tilde{\EB})$ with bounded linear
feedback $\tilde{\Kb}$, belongs to the system class considered in~\citep{XuSallet1996siam}
and~\citep{RiesmeierWoittennek2022arxiv1}.

Now it will be shown that, if Assumption~\ref{hypo:xu_sallet_h3}
holds in terms of $(\desCs{\EA},\ctrlIs{\EB})$,
it also holds in terms of $(\tilde{\EA},\tilde{\EB})$.
This means that it exists an $\tilde{M}\in\Reels^+$ such that
\begin{align*}
\sum_{i=1}^\infty \left|\frac{\tilde{\mb}_i}{\eval - \tilde{\eval}_i}\right|^2 \le \tilde M < \infty,
&& \forall \eval\not\in \tilde D = \bigcup_{i=1}^\infty \tilde D_i,
\end{align*}
with the disks $\{\tilde D_i\}$ (with radius $\frac{\tilde d_i}{3}$) centered at the eigenvalues of $\tilde{\EA}$
and distance $\tilde d_i,\, i\in\Nats$ from $\tilde{\eval}_i\in\sigma(\tilde{\EA})$ to the rest of the spectrum $\sigma(\tilde{\EA})$.
The adjoint of $\tilde{\EA}$ is given by
\begin{align*}
  \tilde{\EA}^* = \begin{pmatrix}
    \desCs[*]{\EA} &  \tilde{\EA}_3^* \\
    0 & \tilde{\EA}_4^*
  \end{pmatrix}: \tilde{\Ss}\supset D(\tilde{\EA}^*)\rightarrow \tilde{\Ss},
\end{align*}
and the eigenvectors $\{\tilde{\evec}_i^*\}$ of $\tilde{\EA}^*$ are given by
\begin{align*}
  \tilde{\evec}_i^* = \left\{
    \begin{array}{l}
        \big((\tilde{\eval}_i \Id_\Ss - \desCs[*]{\EA})^{-1} \tilde{\EA}_3^* \evecII(\tilde{\eval}_i),
        \evecII(\tilde{\eval}_i)\big)
      ,\\
        \big(\evecI(\tilde{\eval}_i), 0\big)
    \end{array}
  \right.
    &&
    \begin{array}{l}
      \vphantom{
        \big((\tilde{\eval}_i \Id_\Ss - \desCs[*]{\EA})^{-1} \tilde{\EA}_3^* \evecII(\tilde{\eval}_i),
        \evecII(\tilde{\eval}_i)\big)}\!\!\!\!\!\!
      i = 1,...,n \\
      \vphantom{
      \big(\evecI(\tilde{\eval}_i), 0\big)}\!\!\!\!\!\!
      i > n,
    \end{array}
\end{align*}
with $\tilde{\eval}_1,...,\tilde{\eval}_n\in\sigma(\tilde{\EA}_4)$, the eigenvector $\evecII(\tilde{\eval}_i)$ of $\tilde{\EA}^*_4$
and the eigenvector $\evecI(\tilde{\eval}_i)$ of $\desCs[*]{\EA}$.
Now $\tilde{\mb}_i = \dualf{\tilde{\EB}}{\tilde{\evec}^*_i}{\tilde{\EA}^*}$, $i\ge 1$ can be determined,
and the sum can be majorized by
\begin{align*}
  \sum_{i=1}^\infty \left|\frac{\tilde{\mb}_i}{\eval - \tilde{\eval}_i}\right|^2
  \le
  \sum_{i=1}^n \left|\frac{\tilde{\mb}_i}{\frac{\tilde{d}_i}{3}}\right|^2
  +
  \sum_{i=1}^\infty \left|\frac{\desCs{\mb}_i}{\eval - \desCs{\eval}_i}\right|^2
  \le \tilde M < \infty,
\end{align*}
with $\tilde{M} = 9\sum_{i=1}^n \left|\frac{\tilde{\mb}_i}{\tilde{d}_i}\right|^2 + M$
and $M$ from Assumption~\ref{hypo:xu_sallet_h3}.

Lemma~3.1 and Lemma~3.2 of \citep{RiesmeierWoittennek2022arxiv1}
are formulated in terms of $(\desCs{\EA},\ctrlIs{\EB},-\Kb;\ctrlIs{\EA})$.
Now they can be applied in terms of $(\tilde{\EA},\tilde{\EB},\tilde{\Kb};\clS{\EA})$,
which completes the proof.
\end{proof}

\subsection{Numerical computation of the eigenvalues}
\label{sec:approx_ev}
The eigenproblem $\clS{\EA}\evec = \eval\evec$, $\evec\in\tilde{\Ss}$ is considered.
To give further insights into the
structure of the eigenproblem,
the eigenvector $\evec=(\hat\evec,\check\evec)$ is decomposed
into $\hat\evec\in\Ss$ and $\check\evec\in\Compl^n$:
\begin{subequations}
\begin{align}
  \label{eq:num_comp_ep_infinite_part}
  \eval\hat\evec &= \BA\hat\evec, && \ctrlIs{\BR}\hat\evec = \MK_n^\Transpose\check\evec \\
  \label{eq:num_comp_ep_finite_part}
  \eval\check\evec &=
  \clS{\EA}_4\check\evec
  + \clS{\EA}_3\hat\evec.
\end{align}
\end{subequations}
Using a solution\footnote{Depending on $\BA$ this solution can be computed by using standard initial- or boundary-value problem solver.
If no solution exists, $\eval$ is not an eigenvalue.} $\hat\evec=\evecI(\eval)$ of $\BA\hat\evec = \eval\hat\evec$,
it becomes clear that $\eval$ is an eigenvalue of $\clS{\EA}$ if the overdetermined
system of equations
\begin{align}
  \label{eq:num_comp_foa_char_eq1}
  \underbrace{
  \begin{pmatrix}
   \MK_n^\Transpose  \\
   \Resol(\clS{\EA}_4, \eval)
  \end{pmatrix}}_{M(\eval)\in\Compl^{n+1\times n}}
  \check\evec
  =
  \underbrace{
  \begin{pmatrix}
    \ctrlIs{\BR} \\ \clS{\EA}_3
  \end{pmatrix}\evecI(\eval)}_{N(\eval)\in\Compl^{n+1}}
\end{align}
has a solution $\check\evec$.
Hence, a characteristic equation of the eigenproblem is
\begin{align*}
  \|M(\eval)M(\eval)^+N(\eval) - N(\eval)\| = 0,
\end{align*}
where $M(\eval)^+$ is the Moore-Penrose inverse of $M(\eval)$.

Another numerically more robust approach consists in rearranging \eqref{eq:num_comp_ep_finite_part} to
  $\check\evec = \Resol(\clS{\EA}_4,\eval) \clS{\EA}_3 \evecI(\eval), \, \eval\not\in\sigma(\clS{\EA}_4)$,
and derive a characteristic equation from the \gls{bc}
of~\eqref{eq:num_comp_ep_infinite_part}:
\begin{align}
  \label{eq:char_eq2}
  \big(\ctrlIs{\BR} - \MK_n^\Transpose \Resol(\clS{\EA}_4,\eval) \clS{\EA}_3\big) \evecI(\eval)=0, && \eval\not\in\sigma(\clS{\EA}_4).
\end{align}
The roots of this characteristic equation represent the infinite part of the point spectrum $\sigma_p(\clS{\EA})$.
Apart from this infinite part, the elements of the finite set $\sigma(\clS{\EA}_4)$ may also belong to
$\sigma_p(\clS{\EA})$. Therefore, it has to checked additionally, for each $\eval\in\sigma(\clS{\EA}_4)$,
whether \eqref{eq:num_comp_foa_char_eq1} has a solution $\check\evec$.

\subsubsection{Application to the example system.}
\label{sec:application}
\begin{table}[hb]
  \begin{center}
  \caption{Parameters for system~\eqref{eq:ex_sys}, used in Section~\ref{sec:approx}, with $\tau=(\ca\cb)^{-\frac{1}{2}}$.}
  \label{tab:params}
  \begin{tabular}{c|c|c|c|c|c|c}
$\ca$ $\,\,$ & $\,\,$ $\cb$ $\,\,$ & $\,\,$ $\cc$ $\,\,$ & $\,\,$ $\muc$ $\,\,$ & $\,\,$ $\kappac$ $\,\,$ & $\,\,$ $\muo$ $\,\,$ & $\,\,$ $\kappao$\\
\hline
\rule{0pt}{8pt}
$11$ $\,\,$ & $\,\,$ $21$ $\,\,$ & $\,\,$ $31$ $\,\,$ & $\,\,$ $e^{-20\tau}$ $\,\,$ & $\,\,$ $15$ $\,\,$ & $\,\,$ $e^{-60\tau}$ $\,\,$ & $\,\,$ $35$
  \end{tabular}
  \end{center}
\end{table}
For the example system~\eqref{eq:ex_sys} with controller and observer
according to Sections~\ref{sec:late_lumping_coc}--\ref{sec:late_lumping_ooc}, and parameters from Table~\ref{tab:params}, the spectrum of the closed-loop
system~\eqref{eq:closed_loop_system} will be computed. To this end, the second of the above-described approaches will be employed,
which is based on the characteristic equation~\eqref{eq:char_eq2}.
Fig.~\ref{fig:obs} shows the closed-loop spectrum for different approximation orders.
Apparently the spectrum $\sigma(\clS{\EA})$ converge to the desired spectra $\sigma(\desCs{\EA})$ and $\sigma(\desOs{\EA})$.
Of course, a proof of spectral convergence, as provided for feedback approximation
in~\cite[Theorem~3.4]{RiesmeierWoittennek2022arxiv1}, remains open for the proposed approximation scheme.
Nevertheless, the obtained numerical results, together with Riesz spectral property the closed-loop system (cf. Theorem~\ref{thm:foa_generation}),
suggests that exponential stability with a certain stability margin can be ensured with the proposed scheme. Furthermore, according to Fig.~\ref{fig:obs}
the required approximation orders are rather small.

\begin{figure}
\begin{center}
\includegraphics[width=\linewidth]{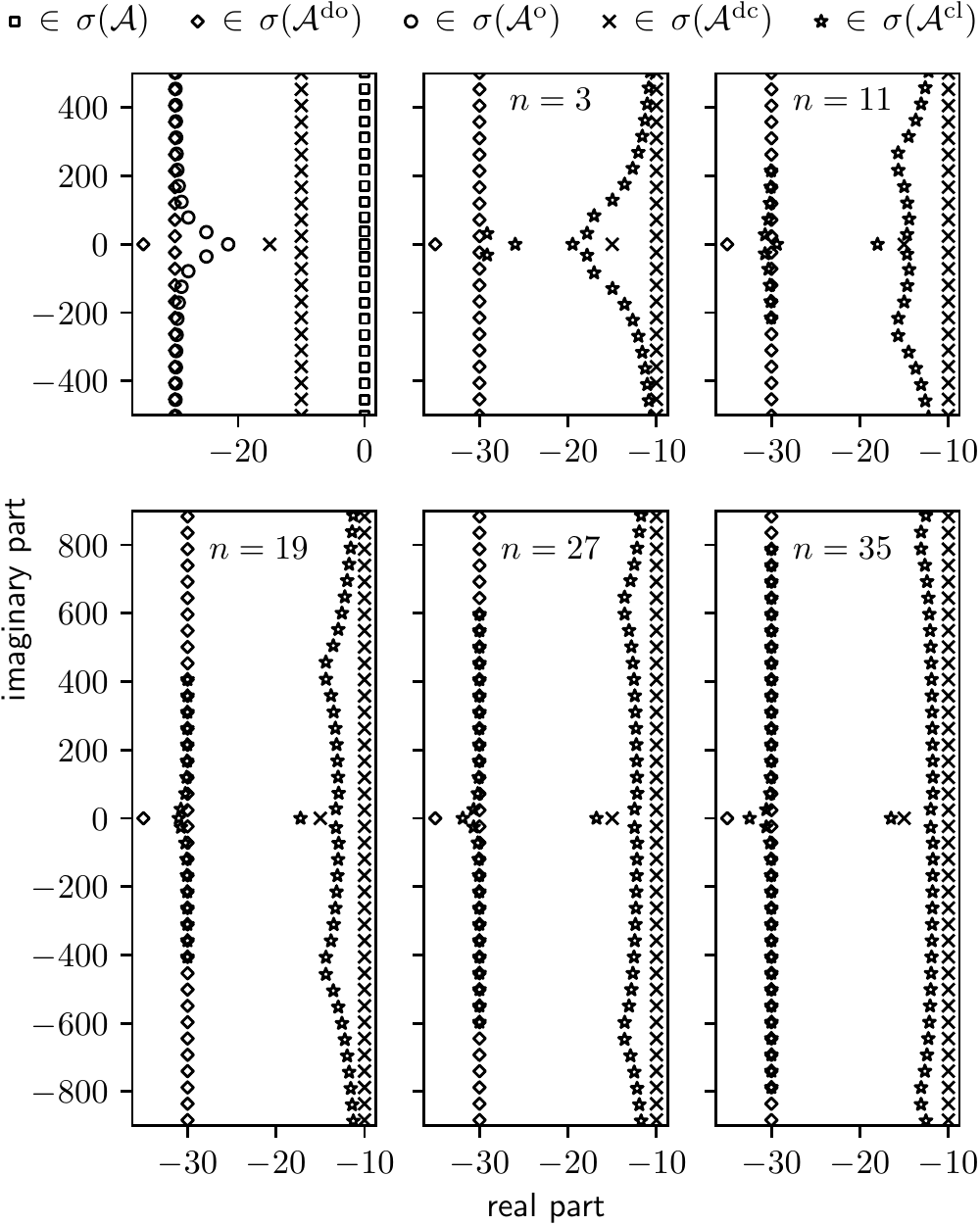}
\end{center}
\caption{Different spectra for of the example system~\eqref{eq:ex_sys}. Upper
left: $\sigma(\EA)$, $\sigma(\desCs{\EA})$, $\sigma(\desOs{\EA})$,
$\sigma(\obsIs{\EA})$.  Other plots: $\sigma(\clS{\EA})$ for different
approximation orders, together with desired closed-loop spectra
$\sigma(\desCs{\EA})$, $\sigma(\desOs{\EA})$.}
\label{fig:obs}
\end{figure}
\section{Conclusion}
\label{sec:conclusion}

An observer approximation scheme is proposed, which avoids possibly existing
deviations in the output equation of standard modal approximation schemes.
By combining this observer with late-lumping state feedback and late-lumping observer output injection,
both described in \citep{RiesmeierWoittennek2022arxiv1}, an approximation scheme for \gls{fosf}
is established, which can be applied to the considered class of \gls{bcs}.
For the closed-loop system operator, the \gls{rs} property is proven.
Therefore, the spectrum-determined growth condition holds.
In contrast to previous results from \citep{Curtain1984siam,Deutscher2013IJC,GrueneMeurer2022},
the results not only hold for analytic systems but for rather general system operators
satisfying the Assumptions~\ref{hypo:ad_spec} and \ref{hypo:xu_sallet_h3}. In particular, this includes certain hyperbolic systems.

\end{document}